\documentclass{amsart}

\usepackage[pdftex]{hyperref}

\hypersetup{ colorlinks=true, linkcolor=blue, filecolor=magenta, urlcolor=cyan, }

\hypersetup{
     colorlinks=true,       
    linkcolor=red,          
    citecolor=blue,        
    filecolor=cyan,         
    urlcolor=magenta        
}

\newcommand{\RR}{\mathbb{R}} 
\newcommand{\CC}{\mathbb{C}} 

\newcommand{\UB}{\mathbb{B}}

\newcommand{\RE}{\mathrm{Re}\,}

\newcommand{\Aut}{\mathrm{Aut}}

\newcommand{\Ric}{\mathrm{Ric}}

\newcommand{\paren}[1]{\left(#1\right)}
\newcommand{\abs}[1]{\left\lvert#1\right\rvert}
\newcommand{\norm}[1]{\left\|#1\right\|}
\newcommand{\set}[1]{\left\{#1\right\}}

\newcommand{\pd}[2]{\frac{\partial#1}{\partial#2}}

\newcommand{\pdl}[2]{\partial#1/\partial#2}

\newcommand{\ind}[4]{{#1}^{\phantom{#2}#3}_{#2\phantom{#3}#4}}

\newtheorem{theorem}{Theorem}[section]

\newtheorem{lemma}[theorem]{Lemma}
\newtheorem{proposition}[theorem]{Proposition}
\theoremstyle{definition}

\theoremstyle{remark}
\newtheorem{remark}[theorem]{Remark}

\newcommand{\KE}{K\"ahler-Einstein }

\def\im{i}
\def\KE{K\"ahler-Einstein }

\numberwithin{equation}{section}

\title[Existence of a complete holomorphic vector field]{Existence of a complete holomorphic vector field  via the K\"ahler-Einstein metric}

\author{Young-Jun Choi}
\address{Department of Mathematics, 
Pusan National University,
2, Busandaehak-ro 63beon-gil,
Geumjeong-gu, Busan, 46241, Republic of Korea}
\email{youngjun.choi@pusan.ac.kr}

\author{Kang-Hyurk Lee}
\address{Department of Mathematics and Research Institute of Natural Science, 
Gyeongsang National University, 
Jinju, Gyeongnam, 52828, 
The Republic of Korea}
\email{nyawoo@gnu.ac.kr}



\thanks{The research of first and second named authors  was supported by the National Research Foundation of Korea (NRF) grant funded by the Korea government (No. 2018R1C1B3005963, No. NRF-2019R1F1A1060891).}

\begin{document}

\maketitle

\begin{abstract}
In this paper, we study the existence of a complete holomorphic vector fields on a strongly pseudoconvex complex manifold admitting a negatively curved complete K\"ahler-Einstein metric and a discrete sequence of automorphisms. Using the method of potential scaling, we will show that there is a potential function of the K\"ahler-Einstein metric whose differential has a constant length. Then we will construct a complete holomorphic vector field from the gradient vector field of the potential function.
\end{abstract}

\section{Introduction}


A fundamental problem in Several Complex Variables is to classify bounded pseudoconvex domains in the complex Euclidean space with a noncompact automorphism group, especially with a compact quotient. A typical result  is due to B.~Wong's theorem in \cite{Wong}: \textit{a bounded strongly pseudoconvex domain in $\CC^n$ with a noncompact automorphism group is biholomorphic to the unit ball $\UB^n=\set{z\in\CC^n:\norm{z}<1}$.} J.P.~Rosay~\cite{Rosay} generalized Wong's theorem: \textit{a bounded domain with an automorphism orbit accumulating at a strongly pseudoconvex boundary point is biholomorphically equivalent to the unit ball.} This implies that  the unit ball is the biholomorphically unique, smoothly bounded domain with a compact quotient.  There have been several generalizations  of the Wong-Rosay theorem. For instance, in Gaussier-Kim-Krantz \cite{GKK}, the unit ball is also characterized among  complex manifolds with strongly pseudoconvex boundary. We will consider in this paper the case of complex manifolds without boundary.

Another important work in this study is due to S.~Frankel~\cite{Frankel}: \textit{a bounded convex domain  with a compact quotient  is symmetric}. A key point in Frankel's work is to show the existence of an $1$-parameter family of automorphisms using the scaling method (for the scaling method, see \cite{KTK1990}).  In  his paper~\cite{LKH}, the second named author of this paper introduced a method of potential scaling for bounded pseudoconvex domains in the complex Euclidean spaces. This method is to rescale a canonical potential function of the \KE metric  by holomorphic automorphisms and then to construct a certain class of potential functions as a rescaling limit. If a rescaling limit satisfies a specified condition, there is  an $1$-parameter family of automorphisms. In this paper, we will generalize this method to a complex manifold with a negatively curved complete \KE metric.

\medskip

Let $X^n$ be a complex manifold of dimension $n$. The \emph{automorphism group} of $X$, denoted by $\Aut(X)$, is the set of self-biholomorphisms of $X$ under the law of the mapping composition.  Throughout this paper, the \emph{negatively curved complete \KE metric} (simply \emph{\KE metric}) of $X$ means a complete \KE metric of $X$ with Ricci curvature $-(n+1)$, equivalently, a complete K\"ahler metric $\omega$ of $X$ with the normalized Einstein condition
\begin{equation*}
\Ric_\omega = -(n+1)\omega \;.
\end{equation*}
In a remarkable work by Yau~\cite{Yau1978}, every compact complex manifold with a negative anticanonical class admits a negatively curved complete \KE metric. By Cheng-Yau~\cite{Cheng-Yau} and Mok-Yau~\cite{Mok-Yau}, a bounded domain in $\CC^n$ admits a \KE metric if and only if the domain is pseudoconvex.

In case of bounded pseudoconvex domains, the \KE metric has a global potential function. Let $X=\Omega$ be  a bounded pseudoconvex domain in $\CC^n$ with the Euclidean coordinates $z=(z^1,\ldots,z^n)$. The negatively curved complete \KE metric $\omega=\im\sum g_{\alpha\bar\beta}dz^\alpha\wedge dz^{\bar\beta}$ of $\Omega$ has the canonical potential function $\log\det(g_{\alpha\bar\beta})$ in the sense that
\begin{equation*}
\im\partial\bar\partial\log\det(g_{\alpha\bar\beta}) = (n+1)\omega \;.
\end{equation*}
since $\Ric_\omega=-\im\partial\bar\partial\log\det(g_{\alpha\bar\beta})$. If $\Omega$ is strongly pseudoconvex, as shown in Proposition~\ref{prop:boundary behavior} in this paper, the length of $\partial\log\det(g_{\alpha\bar\beta})$ with respect to $\omega$ is continuous up to the boundary:
\begin{equation*}
\lim_{p\to\partial\Omega}\norm{\partial\log\det(g_{\alpha\bar\beta})}_\omega^2(p) = (n+1)^2 \;.
\end{equation*}
If a complex manifold $X^n$ without boundary can be biholomorphically imbedded in $\CC^n$ as a strongly pseudoconvex domain, then this noncompact $X$ also admits a negatively curved complete \KE metric $\omega$ and a global potential function $\varphi:X\to\RR$ with 
\begin{equation*}
\lim_{p\to\infty}\norm{\partial\varphi}_\omega^2(p) = (n+1)^2 \;.
\end{equation*}
Here $p\to\infty$ means that $p$ tends to a point of infinity in the one-point compactification of $X$. More precisely, a sequnce $\{p_j\}$ in a manifold $X$ \emph{converges to the point at infinity}, denoted by $p_n\to\infty$ if for any compact subset $K$ in $X$,  $p_j\in X\setminus K$ for sufficiently large $j$.

The main result of this paper is to show the existence of an $1$-paremeter family of automorphisms in this setting.

\begin{theorem}\label{thm:main thm}
Let $X^n$ be a noncompact complex manifold with the complete K\"ahler $\omega$ with Ricci curvature $-(n+1)$.
Suppose that there exists a global potential function $\varphi:X\to\RR$ in the sense of
\begin{equation*}
\im\partial\bar\partial\varphi = (n+1)\omega
\end{equation*}
If
\begin{enumerate}
\item for any sequence $\{p_j\}$ in $X$ converging to the point at infinity,
\begin{equation*}
\lim_{j\to\infty} \norm{\partial\varphi}_\omega (p_j)=n+1 \; ;
\end{equation*}
\item there are a sequence of automorphisms $\{f_j\}$ and a point $p_0\in X$ such that $f_j(p_0)\to\infty$,
\end{enumerate}
then $X$ admits a nowhere vanishing complete holomorphic vector field.
\end{theorem}

A holomorphic vector field $V$ is a holomorphic section to the $(1,0)$-tangent bundle of $X$. If a holomorphic vector field $V$ is complete, equivalently its real part $\RE V=V+\overline{V}$ is complete, then the flow of $\RE V$ is an $1$-parameter family of holomorphic transformations of $X$.

\medskip

In Section~\ref{sec:convergence}, we introduce the method of potential scaling and prove that there is  a global potential function $\tilde\varphi:X\to\RR$ satisfying $\norm{\partial\tilde\varphi}^2_\omega\equiv(n+1)^2$ (Theorem~\ref{thm:main thm1}). Then we will prove that there is complete holomorphic vector field $V$ tangent to $\tilde\varphi$ in Section~\ref{sec:existence} (Theorem~\ref{thm:main thm2}). In the last section, we will discuss a boundary behavior of canonical potential functions in strongly pseudoconvex domains.


\section{Convergence of K\"ahler potentials}\label{sec:convergence}
In this section, we will introduce the method of potential scaling as in \cite{LKH} and will prove that the manifold $X$ in Theorem~\ref{thm:main thm} admits a global potential function $\tilde\varphi$ such that the length of $\partial\tilde\varphi$  is constant.

\medskip

Let $X^n$ be a $n$-dimensional complex manifold with the complete K\"ahler $\omega$ with Ricci curvature $-(n+1)$ and let $\varphi:X\to\RR$ be a global K\"ahler potential  of $\omega$ in the sense of
\begin{equation*}
\im\partial\bar\partial\varphi = (n+1)\omega \;.
\end{equation*}
Since every holomorphic automorphism $f\in\Aut(X)$ is an isometry of $\omega$,
\begin{equation}\label{eqn:pulling-back potential}
\im\partial\bar\partial f^*\varphi 
	= f^* \im\partial\bar\partial\varphi 
	= (n+1)f^*\omega
	= (n+1)\omega \;,
\end{equation}
so each pulling-back $f^*\varphi = \varphi\circ f$ is  also a potential function. The method of potential scaling is to construct a certain potential function as a limit of sequence of potential functions
\begin{equation}\label{eqn:potential scaling}
\varphi\circ f_j - (\varphi\circ f_j)(p_0)
\end{equation}
for some $f_j\in\Aut(X)$ and $p_0\in X$. We will mainly consider the convergence of the sequence.

 When we define 
\begin{equation*}
\psi=\exp\varphi\;, \quad\text{equivalently,}\quad \varphi=\log\psi 
\end{equation*} 
for convenience, we can write \eqref{eqn:potential scaling} by
\begin{equation*}
\varphi\circ f_j - (\varphi\circ f_j)(p_0) = \log\frac{\psi\circ f_j}{(\psi\circ f_j)(p_0)}
\end{equation*}
so it is sufficient to consider the convergence of $\psi\circ f_j/(\psi\circ f_j)(p_0)$. For the convergence of the sequence, we need the following estimates.

\begin{lemma}\label{lem:basic estimate}
Suppose that there is a constant $C>0$ with
\begin{equation}\label{eqn:boundedness}
\norm{\partial\varphi}_\omega < C \quad\text{on $X$.}
\end{equation}
For any compact subset $K\subset X$ and a point $p_0\in X$, there exists a constant $A=A(K,p)>0$ such that
\begin{equation*}
\frac{1}{A} < \frac{\psi\circ f}{(\psi\circ f)(p_0)} <A \quad\text{on $K$}
\end{equation*}
for any $f\in\Aut(X)$.
\end{lemma}
\begin{proof}
The automorphism $f$ is isometric with respect to $\omega$ so that $\norm{\partial(\psi\circ f)}_\omega^2
		=\norm{\partial\psi}_\omega^2\circ f$.
Since $\psi$ is nowhere vanishing on $X$, we have $\partial\psi=\psi (\partial\log\psi)=\psi\partial\varphi$; hence
\begin{equation}\label{eqn:basic inequality}
\norm{\partial(\psi\circ f)}_\omega^2
	= \norm{\partial\psi}_\omega^2\circ f
	= (\psi\circ f)^2 \paren{\norm{\partial\varphi}_\omega^2\circ f} <C^2 (\psi\circ f)^2
\end{equation}
by \eqref{eqn:boundedness}.
When we let 
\begin{equation*}
\sigma_{f}=\frac{\psi\circ f}{(\psi\circ f)(p_0)}
\end{equation*}
for the convenience, the inequality \eqref{eqn:basic inequality}  implies that
\begin{equation*}
\norm{\partial\sigma_{f}}_\omega^2
	= \frac{1}{(\psi\circ f)^2(p_0)}\norm{\partial(\psi\circ f)}_\omega^2
	\leq  C^2(\sigma_{f})^2 
	\;.
\end{equation*}
For a unit speed curve $\gamma:(-R,R)\to\Omega$ with respect to $\omega$ with $\gamma(0)=p_0$, this inequality can be written by $\abs{(\sigma_{f}\circ\gamma)'(t)}\leq C\abs{(\sigma_{f}\circ\gamma)(t)}$. Since $\sigma_{f}(p_0)=1$, Gronwall's inequality gives
\begin{equation*}
e^{-C t}\leq\abs{(\sigma_{f}\circ\gamma)(t)}\leq e^{Ct} \;.
\end{equation*}
As a conclusion, we have that for a point $p\in X$ with $d_\omega(p_0,p)<R$ where $d_\omega$ is the distance associated to $\omega$, we get
\begin{equation*}
e^{-CR}\leq\sigma_{f}(p)\leq e^{CR} \;.
\end{equation*}
This is independent of the choice of $f\in\Aut(X)$. This completes the proof.
\end{proof}

Then we have the convergence of the potential scaling.
\begin{lemma}\label{lem:precompact}
Assume \eqref{eqn:boundedness}. Then for any sequence $\set{f_j}$ of automorphisms of $X$ and a point $p_0\in X$, the sequence of potentials 
\begin{equation*}
\set{ \log\frac{\psi\circ f_j}{(\psi\circ f_j)(p_0)}}
\end{equation*}
has a convergent subsequence in the local $C^\infty$ topology, so a limit is also a potential function.
\end{lemma}
\begin{proof}
Let $Q_j:X\to\RR$ be a positively-valued function defined by
\begin{equation*}
Q_j
	=\frac{\psi\circ f_j}{(\psi\circ f_j)(p_0)}\frac{\psi(p_0)}{\psi} \;.
\end{equation*}
For each compact subset $K$ of $X$, Lemma~\ref{lem:basic estimate} implies that there is a constant $A_K>0$ such that
\begin{equation*}
\frac{1}{A_K}\frac{\psi(p_0)}{\psi} 
	< Q_j 
	<A_K\frac{\psi(p_0)}{\psi}
\end{equation*}
for any $j$. Moreover for the positive constant $B_K=\sup_K \psi$ and $C_K=\inf_K \psi$, we have the uniform estimate
\begin{equation}\label{eqn:second estimate}
0
	<\frac{1}{A_K}\frac{\psi(p_0)}{B_K} 
	< Q_j <A_K\frac{\psi(p_0)}{C_K} 
	\quad\text{on $K$.}
\end{equation}

\medskip

Now we consider the convergence of $\set{Q_j}$. Equation~\eqref{eqn:pulling-back potential} can be written by $\im\partial\bar\partial\log (\psi\circ f_j )=(n+1)\omega$ so that
\begin{equation*}
\im\partial\bar\partial \log Q_j 
	= \im\partial\bar\partial \log (\psi\circ f_j) -\im\partial\bar\partial\log \psi 
		=0 \;.
\end{equation*}
This means $\log Q_j$ is pluriharmonic. By the $\bar\partial$-Poincar\'e Lemma, each $Q_j$ is locally an absolute square of a holomorphic function.  For a small coordinate neighborhood $\mathcal{U}$, we can take a holomorphic function $\eta_j:\mathcal{U}\to\CC$ with $Q_j=\abs{\eta_j}^2$ on $\mathcal{U}$. From Inequality~\eqref{eqn:second estimate}, we may assume that $\abs{\eta_j}$ is pinched by two positive constants on $\mathcal{U}$ so   $\set{\eta_j}$ is has a convergent subsequence in the uniform convergence.  Hurwitz's theorem also says that the limit is nowhere vanishing on $\mathcal{U}$. Since $(X,\omega)$ is complete, we can consider exhaustion by compact subsets of $X$ so the diagonal procedure make us to $\set{Q_j}$ converges subsequentially to a nowhere vanishing function $\widetilde Q:X\to\RR$ in the local $C^\infty$ topology.


As a conclusion, we have a subsequential limit
\begin{equation*}
\frac{\psi\circ f_j}{(\psi\circ f_j)(p_0)} 
	= \frac{\psi}{\psi(p_0)} Q_j 
	\to \frac{\psi}{\psi(p_0)} \widetilde Q
\end{equation*}
in the local $C^\infty$ topology. Since $\im\partial\bar\partial\log\widetilde Q = 0$, we have
\begin{equation*}
\im\partial\bar\partial\tilde\varphi 
	= \im\partial\bar\partial\varphi = (n+1)\omega
\end{equation*}
where 
\begin{equation*}
\tilde\varphi 
	= \log \frac{\psi}{\psi(p)} \widetilde Q
	\;.
\end{equation*}
This completes the proof.
\end{proof}

Now we suppose that there is a constant $C_0>0$ suth that the value of $\norm{\partial\varphi}_\omega$ at the point at infinity is always $C_0$:
\begin{equation*}
\lim_{p\to\infty} \norm{\partial\varphi}_\omega^2(p)=C_0
\end{equation*}
Then the length of $\partial\varphi$ is globally bounded: $\norm{\partial\varphi}_\omega<C$ on $X$ for some $C>0$.
If there is a sequence of automorphisms $\set{f_j}$ and a point $p_0\in X$ such that $f_j(p_0)\to\infty$, Lemma~\ref{lem:precompact} makes us to get a potential rescaling limit
\begin{equation*}
\varphi\circ f_j - (\varphi\circ f_j)(p_0) \to \tilde\varphi
\end{equation*}
where $\tilde\varphi$ satisfies $\im\partial\bar\partial\tilde\varphi = (n+1)\omega$. For any $p\in X$, we have $f_j(p)\to\infty$ by the completeness of $X$ and
\begin{multline*}
\norm{\partial\tilde\varphi}_\omega(p) 
	=\lim_{j\to\infty} \norm{\partial\paren{\varphi\circ f_j - (\varphi\circ f_j)(p_0)}}_\omega(p) 
	\\
	= \lim_{j\to\infty}\norm{\partial(\varphi\circ f_j) }_\omega(p) 
	=\lim_{j\to\infty} \norm{\partial\varphi }_\omega(f_j(p)) 
\end{multline*}
by the convergence of the potential scaling. This means that $\norm{\partial\tilde\varphi}_\omega\equiv C_0$.

\medskip

As a conclusion, we have
\begin{theorem}\label{thm:main thm1}
Let $X^n$ be a noncompact complex manifold with the complete K\"ahler $\omega$ with Ricci curvature $-(n+1)$.
Suppose that there exists $\varphi:X\to\RR$ such that $\im\partial\bar\partial\varphi = (n+1)\omega$ and 
\begin{equation*}
\lim_{p\to\infty} \norm{\partial\varphi}_\omega(p) =C_0 \;.
\end{equation*}
for some constant $C_0$. If $f_j(p_0)\to\infty$ for some $f_j\in\Aut(X)$ and $p_0\in X$, then there is $\tilde\varphi:X\to\RR$ with 
\begin{equation*}
\im\partial\bar\partial\tilde\varphi = (n+1)\omega\quad\text{and}\quad \norm{\partial\tilde\varphi}_\omega \equiv C_0\;.
\end{equation*}
\end{theorem}

This theorem implies that the space $X$ in Theorem~\ref{thm:main thm} admits a global potential function $\tilde\varphi$ such that $\norm{\partial\tilde\varphi}_\omega \equiv n+1$.


\section{Existence of complete holomorphic vector fields}\label{sec:existence}
In this section, we will study the existence of complete holomorphic vector fields on a negatively curved complete \KE manifold.

\medskip

Let $X^n$ be a $n$-dimensional complex manifold with the complete K\"ahler $\omega$ with Ricci curvature $-(n+1)$ and suppose that $\im\partial\bar\partial\varphi=(n+1)\omega$ on $X$. 
In a local coordinate function $z=(z^1,\ldots,z^n)$, we can write
\begin{equation*}
\omega=\im g_{\alpha\bar\beta} dz^\alpha\wedge dz^{\bar\beta} \;.
\end{equation*}
We denote the complex conjugate of a tensor by taking the bar on the indices, that is, $\overline{z^\alpha} = z^{\bar\alpha}$, $\overline{g_{\alpha\bar\beta}} = g_{\bar\alpha\beta}$ and so on. We will also use the matrix of the  \KE metric $(g_{\alpha\bar\beta})$ and its inverse matrix $(g^{\bar\beta\alpha})$ to raise and lower indices: $\theta_\alpha=g_{\alpha\bar\beta}\theta^{\bar\beta}$, $\ind{R}{\beta}{\alpha}{\mu\bar\nu}=g^{\bar\gamma\alpha}R_{\beta\bar\gamma\mu\bar\nu}$. The Greek index $\alpha,\beta,\ldots$ runs from $1$ to $n$ and the summation convention for duplicated indices is always assumed.

Then $\norm{\partial\varphi}_\omega^2$ can be written by
\begin{equation*}
\norm{\partial\varphi}_\omega^2 
	= \norm{\varphi_\alpha dz^\alpha}_\omega^2 
	=\varphi_\alpha \varphi_{\bar\beta} g^{\alpha\bar\beta} = \varphi_\alpha\varphi^{\alpha}
\end{equation*}
where $\varphi_\alpha=\pdl{\varphi}{z^\alpha}$ and $\varphi_{\bar\beta}=\pdl{\varphi}{z^{\bar\beta}}$.

We denote the K\"ahler connection with respect to $\omega$ by $\nabla$ and denote the covariant derivative with repsect to $\pd{}{z^\alpha}$ and $\pd{}{z^{\bar\beta}}$ by $\nabla_\alpha, \nabla_{\bar\beta}$, or we will use the semi-colon notation.

\medskip

In order to show the existence of complete holomorphic vector fields, we need the following PDE equation.

\begin{proposition}\label{P:PDE}
The norm $u=\norm{\partial\varphi}^2_\omega$ satisfies the following PDE:
\begin{equation*}
	\Delta_\omega u
	=
	\norm{\nabla'^2 \varphi}_\omega^2
	+
	n(n+1)^2-(n+1)u
\end{equation*}
where $\nabla'$ is the $(1,0)$-part of $\nabla$ and $\Delta_\omega$ is the Laplace-Beltrami operator with non-positive eigenvalues.
\end{proposition}

\begin{proof}
In the local coordinates $z=(z^1,\ldots,z^n)$, the identity $\im\partial\bar\partial\varphi=(n+1)\omega$ implies that $\varphi_{\alpha;\bar\beta} = \nabla_{\bar\beta}\varphi_\alpha = (n+1)g_{\alpha\bar\beta}$. Moreover we have
\begin{equation}\label{eqn:identity1}
	\varphi_{\alpha;\lambda\bar\mu}
	=
	\varphi_\beta\ind R\alpha\beta{\lambda\bar\mu}
	\;, \quad
	\varphi_{\bar\beta;\lambda\bar\mu}
	=0
\end{equation}
where $\ind R\alpha\beta{\lambda\bar\mu}$ stands for the curvature tensor: $(\nabla_\lambda\nabla_{\bar\mu}-\nabla_{\bar\mu}\nabla_\lambda)\pdl{}{z^\alpha} =\ind R\alpha\beta{\lambda\bar\mu}\pdl{}{z^\beta}$.
In fact, since $\nabla$ is a K\"ahler connection of $\omega$, we have
\begin{align*}
\nabla_{\bar\mu}\nabla_\lambda \varphi_\alpha
&=
\nabla_\lambda\nabla_{\bar\mu} \varphi_\alpha
+
\varphi_\beta\ind R\alpha\beta{\lambda\bar\mu}
=
\nabla_\lambda \varphi_{\alpha;\bar\mu}
+
\varphi_\beta\ind R\alpha\beta{\lambda\bar\mu} \\
&=
(n+1)\nabla_\lambda g_{\alpha\bar\mu}
+
\varphi_\beta\ind R\alpha\beta{\lambda\bar\mu}
=
\varphi_\beta\ind R\alpha\beta{\lambda\bar\mu} \;,
\\
\nabla_{\bar\mu}\nabla_\lambda \varphi_{\bar\beta}
&=
\nabla_{\bar\mu} \varphi_{\bar\beta;\lambda}
= (n+1) \nabla_{\bar\mu} g_{\bar\beta\lambda}
=0 \;.
\end{align*}
Since $u=\varphi_\gamma \varphi^\gamma$, 
\begin{align*}
	\Delta_\omega u
	&=
	g^{\alpha\bar\beta}
	\nabla_{\bar\beta}\nabla_\alpha
	\paren{\varphi_\gamma \varphi^\gamma}
	\\
	&=
	g^{\alpha\bar\beta}
	\nabla_{\bar\beta}
	\paren{
		\varphi_{\gamma;\alpha}\varphi^\gamma
		+
		\varphi_\gamma\ind{\varphi}{}{\gamma}{;\alpha} 
	}
	\\
	&=
	g^{\alpha\bar\beta}
	\paren{
		\varphi_{\gamma;\alpha\bar\beta} \varphi^\gamma
		+
		\varphi_{\gamma;\alpha}\ind{\varphi}{}{\gamma}{;\bar\beta}
		+
		\varphi_{\gamma;\bar\beta}\ind{\varphi}{}{\gamma}{;\alpha} 
		+
		\varphi_\gamma\ind{\varphi}{}{\gamma}{;\alpha\bar\beta} 
	} 
	\;.
\end{align*}
Using \eqref{eqn:identity1}, we have $\varphi_{\gamma;\alpha\bar\beta} \varphi^\gamma
	=\varphi_\delta\ind R\gamma\delta{\alpha\bar\beta}\varphi^\gamma
	=\varphi^{\bar\delta} R_{\gamma\bar\delta\alpha\bar\beta}\varphi^\gamma$. Since $R_{\gamma\bar\delta\alpha\bar\beta}g^{\alpha\bar\beta}=R_{\gamma\bar\delta}=-(n+1)g_{\gamma\bar\delta}$ from the Einstein condition, we have 
\begin{equation*}	
g^{\alpha\bar\beta}\varphi_{\gamma;\alpha\bar\beta} \varphi^\gamma = -(n+1)\varphi^{\bar\delta}\varphi^{\gamma}g_{\gamma\bar\delta} = -(n+1)u
\end{equation*}
The second and third terms above can be written by 
\begin{align*}
g^{\alpha\bar\beta}\varphi_{\gamma;\alpha}\ind{\varphi}{}{\gamma}{;\bar\beta}
	&=\varphi_{\gamma;\alpha}\varphi^{\gamma;\alpha}
	 = \norm{\nabla'^2 \varphi}_\omega^2 \;,
	 \\
g^{\alpha\bar\beta}\varphi_{\gamma;\bar\beta}\ind{\varphi}{}{\gamma}{;\alpha} 	 
	&= \varphi_{\gamma;\bar\beta}\varphi^{\gamma;\bar\beta}
	= (n+1)^2 g_{\gamma\bar\beta}g^{\gamma\bar\beta}
	= n(n+1)^2 \;.
\end{align*}
The last term is vanishing:  $\varphi_\gamma\ind{\varphi}{}{\gamma}{;\alpha\bar\beta} = \varphi_\gamma g^{\gamma\bar\delta}\varphi_{\bar\delta;\alpha\bar\beta}=0$ from the second equation in \eqref{eqn:identity1}. Therefore we can obtain
\begin{align*}
	\Delta_\omega u
	&=
	-(n+1)u+\varphi_{\alpha;\beta}\varphi^{\alpha;\beta}+n(n+1)^2\;	.
\end{align*}
This completes the proof.
\end{proof}

Now we will show the existence of complete holomorphic vector fields. 
\begin{theorem}\label{thm:main thm2}
Let $\omega$ be the complete \KE metric on $X$ with Ricci curvature $-(n+1)$.
Suppose that there exist a global potential $\varphi$ of $\omega$ in the sense of  $\im\partial\bar\partial\varphi=(n+1)\omega$ such that
\begin{equation*}
\norm{\partial\varphi}_\omega\equiv n+1.
\end{equation*}
Then the vector field 
\begin{equation}\label{eqn:infinitesimal generator}
	V
	=
	\im e^{\frac{\varphi}{n+1}}\mathrm{grad}(\varphi)
\end{equation}
is a complete holomorphic vector field.
\end{theorem}
Here $\mathrm{grad}(\varphi)$ is the $(1,0)$-part of the gradient vector field of $\varphi$ with respect to $\omega$:
\begin{equation*}
	\mathrm{grad}(\varphi)
	=
	\varphi^\alpha\pd{}{z^\alpha}
\end{equation*}
where $\varphi^\alpha=\varphi_{\bar\beta}g^{\bar\beta\alpha}$ (\cite{Jost}). When we denote by $W=\im\mathrm{grad}(\varphi)$, then $V=e^{\frac{\varphi}{n+1}}W$. Since the vector field $W$ is the turned gradient of $\varphi$ by the complex structure,  $W$ is tangent to $\varphi$:
\begin{equation*}
(\RE W) \varphi
	 =  \paren{\im\varphi^\alpha\pd{}{z^\alpha}-\im\varphi^{\bar\alpha}\pd{}{z^{\bar\alpha}}}\varphi
	 =   \paren{\im\varphi^\alpha\varphi_{\alpha}-\im\varphi^{\bar\alpha}\varphi_{\bar\alpha}}
	 =0 \;.
\end{equation*}
So we have 
\begin{equation*}
(\RE W)e^{\frac{\varphi}{n+1}}=0\;.
\end{equation*}
Moreover $W$ has constant length $\norm{W}_\omega=\norm{\partial\varphi}_\omega$, so the real tangent vector field $\RE W$ is complete. Therefore the following lemma implies that $V=e^{\frac{\varphi}{n+1}}W$ is also complete.
\begin{lemma}
Let $Z$ is a complete $(1,0)$ tangent vector field on $X$. If there is a nowhere vanishing smooth function $\rho:X\to\RR$ with $(\RE Z)\rho\equiv 0$, then $\rho Z$ is also complete.
\end{lemma}
\begin{proof}
Take an integral curve $\gamma:\RR\to X$ of $\RE Z$. It satisfies 
\begin{equation*}
(\RE Z)\circ \gamma = \dot\gamma
\end{equation*}
Since $(\RE Z)\rho\equiv 0$, the curve $\gamma$ lies on a level set of $\rho$ so $\rho\circ\gamma\equiv c$ for some  constant $c$. For the curve $\sigma:\RR\to X$ defined by $\sigma(t)=\gamma(ct)$, we have 
\begin{equation*}
(\rho\RE Z)\circ \sigma (t)  = (\rho\RE Z)\circ \gamma (ct)  = c(\RE Z)(\gamma(ct)) = c\dot\gamma(ct) =\dot\sigma(t)
\end{equation*}
This means that $\sigma:\RR\to\Omega$ is the integral curve of $\rho\RE Z$; therefore $\rho\RE Z$ is complete. 
\end{proof}

Now we will prove that $V$ in \eqref{eqn:infinitesimal generator} is holomorphic. The hypothesis and Proposition \ref{P:PDE} imply that $\norm{\nabla'^2\varphi}_\omega^2=\varphi_{\alpha;\beta}\varphi^{\alpha;\beta}=(n+1)^2$. On the other hand, we have
\begin{align*}
	0
	&=
	\partial\paren{\norm{\partial\varphi}_\omega^2}
	=
	(\varphi_\alpha \varphi^\alpha)_{;\gamma}dz^\gamma
	=
	\paren{
		\varphi_{\alpha;\gamma}\varphi^\alpha
		+
		\varphi_\alpha\ind{\varphi}{}{\alpha}{;\gamma}
	}
	dz^\gamma\\
	&=
	\paren{
		\varphi_{\alpha;\gamma}\varphi^\alpha
		+
		\varphi_\alpha\ind{\varphi}{}{}{\bar\beta;\gamma}g^{\bar\beta\alpha}
	}
	dz^\gamma
	=
	\paren{
		\varphi_{\alpha;\gamma}\varphi^\alpha
		+
		(n+1)\varphi_\alpha g_{\gamma\bar\beta}g^{\bar\beta\alpha}
	} 
	dz^\gamma\\
	&=
	\paren{
		\varphi_{\alpha;\gamma}\varphi^\alpha
		+
		(n+1)\varphi_\gamma
	}dz^\gamma.
\end{align*}
It follows that
\begin{equation}\label{E:Key_equation}
\varphi_{\alpha;\gamma}\varphi^\alpha
=
-(n+1)\varphi_\gamma.
\end{equation}
Denote by $\psi=\exp\varphi$.
Recall that $V$ is defined as follows:
\begin{equation*}
V
=
\im V^\alpha\pd{}{z^\alpha}
\quad\text{where}\quad
V^\alpha
=
e^{\frac{\varphi}{n+1}}\varphi^\alpha
=
\psi^{\frac{1}{n+1}}\varphi^\alpha.
\end{equation*}
It follows that
\begin{equation*}
\nabla'' V
=
\im\ind{V}{}{\alpha}{;\bar\beta}\pd{}{z^\alpha}\otimes dz^{\bar\beta}
\end{equation*}
where $\nabla''$ is the $(0,1)$ part of the K\"ahler connection $\nabla$ and
\begin{align*}
	\ind{V}{}{\alpha}{;\bar\beta}
	&=
	\frac{1}{n+1}\psi^{\frac{1}{n+1}}
	\paren{\log\psi}_{\bar\beta}\varphi^\alpha
	+
	\psi^{\frac{1}{n+1}}\ind{\varphi}{}{\alpha}{;\bar\beta}
	\\
	&=
	\frac{1}{n+1}\psi^{\frac{1}{n+1}}
	\varphi_{\bar\beta}\varphi^\alpha
	+
	\psi^{\frac{1}{n+1}}\ind{\varphi}{}{\alpha}{;\bar\beta}.
\end{align*}
A straightforward computation gives that
\begin{align*}
	\norm{\nabla'' V}_\omega^2
	&=
	\frac{1}{(n+1)^2}(\psi^{\frac{1}{n+1}})^2(\varphi_\alpha \varphi^\alpha)^2
	+
	\frac{1}{n+1}(\psi^{\frac{1}{n+1}})^2 \varphi_{\bar\beta}
	\varphi^\alpha\ind{\varphi}{\alpha}{;\bar\beta}{}\\
	&\quad
	+
	\frac{1}{n+1}(\psi^{\frac{1}{n+1}})^2
	\ind{\varphi}{}{\alpha}{;\bar\beta}\varphi^{\bar\beta} \varphi_\alpha
	+
	(\psi^{\frac{1}{n+1}})^2 \varphi^{\alpha;\beta}\varphi_{\alpha;\beta}.
\end{align*}
It follows from Proposition \ref{P:PDE} and \eqref{E:Key_equation} that
\begin{equation*}
\norm{\nabla'' V}_\omega^2
=
(\psi^{\frac{1}{n+1}})^2
\paren{\frac{1}{(n+1)^2}\norm{\partial\varphi}^4_\omega
+
2\frac{-(n+1)}{n+1}\norm{\partial\varphi}^2_\omega
+
(n+1)^2
}=0 \;.
\end{equation*}
This implies that $V$ is holomorphic.

\medskip

Combining Theorem~\ref{thm:main thm1} and Theorem~\ref{thm:main thm2}, we obtain Theorem~\ref{thm:main thm}.

\begin{remark}
Suppose that $\varphi:X\to\RR$ satisfies $\im\partial\bar\partial\varphi=(n+1)\omega$ and 
\begin{equation*}
\norm{\partial\varphi}_\omega\equiv C
\end{equation*}
for some constant $C>0$. The bundle morphism $S:T^{1,0}X\to T^{1,0}X$ defined locally by
\begin{equation*}
S = \ind{\varphi}{}{\alpha}{;\bar\beta}\ind{\varphi}{}{\bar\beta}{;\gamma}\pd{}{z^\alpha}\otimes dz^\gamma 
\end{equation*}
is a positive semidefinite symmetric operator. Equation~\eqref{E:Key_equation} implies that $\mathrm{grad}(\varphi)$ is a field of eigenvectors of $S$ with constant eingenvalue $(n+1)^2$:
\begin{multline*}
S(\mathrm{grad}(\varphi)) 
	= S\paren{\varphi^\gamma \pd{}{z^\gamma} }
	=  \ind{\varphi}{}{\alpha}{;\bar\beta}\ind{\varphi}{}{\bar\beta}{;\gamma}\varphi^\gamma \pd{}{z^\alpha} 
	= -(n+1)\ind{\varphi}{}{\alpha}{;\bar\beta}\varphi^{\bar\beta}\pd{}{z^\alpha} 
	\\
	= (n+1)^2\varphi^\alpha \pd{}{z^\alpha} 
	= (n+1)^2 \mathrm{grad}(\varphi) \;.
\end{multline*}
Since the value $\norm{\nabla'^2 \varphi}_\omega^2 = \varphi^{\alpha;\beta}\varphi_{\beta;\alpha} = \ind{\varphi}{}{\alpha}{;\bar\beta}\ind{\varphi}{}{\bar\beta}{;\alpha}$ coincides with the trace of $S$, we have $\norm{\nabla'^2 \varphi}_\omega^2\geq(n+1)^2$. Then the PDE equation in Proposition~\ref{P:PDE} implies $(n+1)C^2=\norm{\nabla'^2 \varphi}_\omega^2+n(n+1)^2\geq (n+1)^3$ so 
\begin{equation*}
C\geq n+1\;.
\end{equation*}
For the case of the unit ball $\UB^n = \set{z\in\CC^n: \norm{z}<1}$ where $z=(z^1,\ldots,z^n)$ is the standard coordinates and $\norm{\;\cdot\;}$ is the Euclidean norm, the function
\begin{equation*}
\varphi=(n+1)\log \frac{\abs{1+z^1}^{2}}{\paren{1-\norm{z}^2}}
\end{equation*}
satisfies $\im\partial\bar\partial\varphi=(n+1)\omega$ and $\norm{\partial\varphi}_\omega \equiv n+1$. This is an optimal case of the inequality $C\geq n+1$.
\end{remark}


\section{Boundary behavior of the \KE metric on a strongly pseudoconvex domain}
In this section, we shall compute the boundary behavior of $\norm{\partial\varphi}^2_\omega$ on a bounded strongly pseudoconvex domain. First we briefly recall the boundary behavior of the solution of the complex Monge-Ampere equation due to Cheng and Yau~\cite{Cheng-Yau}.

Let $\Omega$ be a smooth bounded strongly pseudoconvex domain in $\CC^n$. Then there exists a defining function $r$ of $\Omega$ satisfying the following conditions:
\begin{itemize}
\item [(\romannumeral1)] $r\in{C^\infty(\overline\Omega})$,
\item [(\romannumeral2)] $\Omega=\{z\in\CC^n:r(z)<0\}$,
\item [(\romannumeral3)] $\partial r\neq0$ on $\partial\Omega$, and
\item [(\romannumeral4)] $\paren{r_{\alpha\bar\beta}} >0$ in $\overline\Omega$.
\end{itemize}
Denote by $w=-\log(-r)$. Then $w$ is a strictly plurisubharmonic function defined in $\Omega$. Easy calculations show that
\begin{equation} \label{E:Levi_form_of_g}
w_{\alpha\bar\beta}
=\frac{r_{\alpha\bar\beta}}{-r}+
\frac{r_\alpha r_{\bar\beta}}{r^2}\;,
\end{equation}
and the inverse is
\begin{equation} \label{E:inverse}
w^{\bar\beta\alpha}
	=(-r)\paren{r^{\bar\beta\alpha}
	+\frac{r^{\bar\beta}r^\alpha}{r-\abs{\partial r}^2}},
\end{equation}
where
\begin{equation*}
(r^{\alpha\bar\beta})=
\paren{r_{\alpha\bar\beta}}^{-1},
\quad
r^\alpha
	=r^{\alpha\bar\beta}r_{\bar\beta},
\quad\text{and} \quad
\abs{\partial r}^2
	= r^{\alpha\bar\beta}r_\alpha r_{\bar\beta}.
\end{equation*}
It is also easy to see that
\begin{equation*}
w^{\bar\beta\alpha}w_\alpha{w}_{\bar\beta}
	=\frac{\abs{\partial r}^2}{\abs{\partial r}^2-r}\le1.
\end{equation*}
Thus the metric $w_{\alpha\bar\beta}$ is a complete K\"{a}hler metric on $\Omega$. The Ricci tensor of $w_{\alpha\bar\beta}$ is given by
\begin{equation}\label{E:compatible}
\begin{aligned}
R_{\alpha\bar\beta} 
	& =-\pd{^2}{z^\alpha\partial{z}^{\bar\beta}} \log\det(w_{\gamma\bar\delta}) \\
	& = -(n+1)w_{\alpha\bar\beta}
	-\frac{\partial^2}{\partial{z^\alpha}\partial{z^{\bar\beta}}}
	\log\det(r_{\gamma\bar\delta})(-r+\abs{\partial r}^2).
\end{aligned}
\end{equation}
If we denote by $F=\log\det(r_{\alpha\bar\beta})(-r+\abs{\partial r}^2)$, then $F$ is a positive smooth function in $\overline\Omega$ satisfying
\begin{equation*}
R_{\alpha\bar\beta}
+(n+1)w_{\alpha\bar\beta}
=
\frac{\partial^2F}{\partial{z^\alpha}\partial{z^{\bar\beta}}}.
\end{equation*}
It is remarkable to note that the function $F$, which measures how far the metric $w_{\alpha\bar\beta}$ from the \KE metric, depends on the defining function $r$. The Cheng-Yau Theorem implies that there exists a solution the following complex Monge-Ampere equation:
\begin{equation} \label{E:Monge-Ampere'}
\begin{aligned}
\det(w_{\alpha\bar\beta}+u_{\alpha\bar\beta})
=e^{(n+1)u}e^F\det(w_{\alpha\bar\beta}) \\
\frac{1}{c}(w_{\alpha\bar\beta})\le(w_{\alpha\bar\beta}+u_{\alpha\bar\beta}) \le c(w_{\alpha\bar\beta}).
\end{aligned}
\end{equation}
It is easy to see that $w_{\alpha\bar\beta}+u_{\alpha\bar\beta}$ is the unique complete \KE metric on $\Omega$. 
In \cite{Fefferman}, Fefferman developed a way to find a good defining function $r$, which is called an approximate solution of the Monge-Ampere equation. This defining function $r$ satisfies that 
\begin{equation*}
F
=
\log\det(r_{\alpha\bar\beta})(-r+\abs{\partial r}^2)
=
O(\abs{r}^{n+1})
\end{equation*}
Using this approximate solution, Cheng and Yau computed the boundary behavior of $u$.

\begin{theorem}[Simple Version \cite{Cheng-Yau}] \label{T:CY2}
Let $\Omega$ be a smooth strongly pseudoconvex domain in $\CC^n$ and let $r$ be a smooth defining function of $\Omega$. Suppose that $F=O(\abs r^{n+1})$ and $u$ is a solution of \eqref{E:Monge-Ampere'}. Then
\begin{equation*}
\abs{D^pu}(x)=O(\abs r^{n+1/2-p-\varepsilon})
\quad\text{for}\quad \varepsilon>0
\end{equation*}
where $\abs{D^pu}(x)$ is the Euclidean length of the $p$-th derivative of $u$.
\end{theorem}
In particular Theorem \ref{T:CY2} says that
\begin{equation} \label{E:boundary}
\abs{u_\alpha}= O(\abs r^{n-1/2-\varepsilon})
\quad\text{and}\quad
\abs{u_{\bar\beta}}= O(\abs r^{n-1/2-\varepsilon})
\end{equation}
for $\varepsilon>0$ and $1\le\alpha,\beta\le{n}$.
Before computing the boundary behavior of $\norm{\partial\varphi}_\omega$, we introduce the following lemma.
\begin{lemma}[\cite{Choi}] \label{L:identity}
There exists a hermitian $n\times{n}$ matrix 
\begin{equation*}
N=(N_{\alpha\bar\beta})
	\in\mathrm{Mat}_{n\times{n}}\paren{C^\infty(\Omega)\cap{C}^{n-3/2-\varepsilon}
	(\overline \Omega)}
\end{equation*}
with $\norm{N}=O(\abs r^{n-3/2-\varepsilon})$ for $\varepsilon>0$, which satisfies that
\begin{equation*}
g^{\bar\beta\alpha}-w^{\bar\beta\alpha}
	=w^{\bar\beta\gamma}N_{\gamma\bar\delta}w^{\bar\delta\alpha}.
\end{equation*}
In particular, $g^{\bar\beta\alpha}\in{C}^\infty(\Omega)\cap{C}^{n-3/2-\varepsilon}(\overline\Omega)$ and $g^{\bar\beta\alpha}=O(\abs r)$ for $\varepsilon>0$.
\end{lemma}

Now we consider the boundary behavior of $\norm{\partial\varphi}_\omega^2$ near the boundary.
Note that $\varphi=(n+1)g=(n+1)(w+u)$.
It follows from Lemma \ref{L:identity} that
\begin{align*}
\norm{\partial\varphi}^2_\omega
&=
\varphi_\alpha \varphi_{\bar\beta}g^{\alpha\bar\beta}
=
(n+1)^2(w+u)_\alpha(w+u)_{\bar\beta}g^{\alpha\bar\beta}
\\
&=
(n+1)^2(w+u)_\alpha(w+u)_{\bar\beta}
\paren{w^{\bar\beta\alpha}+w^{\bar\beta\gamma}N_{\gamma\bar\delta}w^{\bar\delta\alpha}}
\\
&=
(n+1)^2
\Big(w_\alpha w_{\bar\beta}w^{\alpha\bar\beta}
	+
	w_\alpha u_{\bar\beta}w^{\bar\beta\alpha}
	+
	u_\alpha w_{\bar\beta}w^{\bar\beta\alpha}
	\\
	&\qquad\qquad\qquad\qquad\qquad+
	(w+u)_\alpha(w+u)_{\bar\beta}
	w^{\bar\beta\gamma}N_{\gamma\bar\delta}w^{\bar\delta\alpha}
\Big) \;.
\end{align*}
It follows from \eqref{E:Levi_form_of_g}, \eqref{E:inverse} and \eqref{E:boundary} that
\begin{align*}
w_\alpha&=O(\abs r^{-1})\;, &
w^{\bar\beta\alpha}&=O(\abs r)\;, \\
u_\alpha&=O(\abs r^{n-1/2-\varepsilon}) \;,
&
N^{\gamma\bar\delta}
&=O(\abs r^{n-3/2-\varepsilon})\;,
\end{align*}
thus we have
\begin{equation*}
\norm{\partial\varphi}^2_\omega
=
(n+1)^2\frac{\abs{\partial r}^2}{\abs{\partial r}^2-r}
+O(\abs r).
\end{equation*}
In particular, we have
\begin{proposition}\label{prop:boundary behavior}
Let $\Omega$ be a bounded strongly pseudoconvex domain with smooth boundary.
Let $\omega=\im\sum g_{\alpha\bar\beta}dz^\alpha\wedge dz^{\bar\beta}$ be the unique complete \KE metric on $\Omega$. Then
\begin{equation*}
\norm{\partial\varphi}^2_\omega\rightarrow(n+1)^2
\quad
\text{as $p\rightarrow\partial\Omega$,}
\end{equation*}
where $\varphi=\log\det(g_{\alpha\bar\beta})$.
\end{proposition}

\begin{remark}
Every computation is easily generalized to a bounded strongly pseudoconvex domain with smooth boundary in a complex manifold by \cite{Coevering}.
Moreover, it is also localized. More precisely, near a strognly pseudoconvex boundary point of a bounded pseudoconex domain, one can obtain the same conclusion of this section
(\cite{Gontard}).
\end{remark}


\end{document}